\documentclass[12pt]{article}

\usepackage{amssymb}%
\usepackage{amsmath}%
\usepackage{amsthm}%
\usepackage{xcolor}%

\allowdisplaybreaks%

{\theoremstyle{plain}%
 \newtheorem{theorem}{Theorem}

 \newtheorem{lemma}{Lemma}% 
}
{\theoremstyle{remark}

}
{\theoremstyle{definition}

}

\begin{document}

\begin{center}
{\Large Reduced complexities for sequences over finite alphabets}

 \ 

{\textsc{John M. Campbell}, \textsc{James Currie}, and \textsc{Narad Rampersad}} 

 \ 

\end{center}

\begin{abstract}
 Letting $w$ denote a finite, nonempty word, let $\text{red}(w)$ denote the word obtained from $w$ by replacing every subword $s$ of 
 $w$ of the form $cc \cdots c$ for a given character $c$ (such that there is no character immediately to the left or right of $s$ equal to 
 $c$) with $c$. Complexity functions for infinite words play important roles within combinatorics on words, and this leads us to introduce 
 and investigate variants of the factor and abelian complexity functions using the given reduction operation. By enumerating words $v$ 
 and $w$ of a given length $n \geq 0$ and associated with an infinite sequence over a finite alphabet such that $\text{red}(v)$ and 
 $\text{red}(w)$ are equal or otherwise equivalent in some specified way, by analogy with the factor and abelian complexity functions, this 
 may be seen as producing simplified versions of previously introduced complexity functions. We prove a recursion for the reduced factor 
 complexity function $\rho_{\text{{\bf t}}}^{\text{red}}$ for the Thue--Morse sequence $\text{{\bf t}}$, giving us that 
 $(\rho_{\text{{\bf t}}}^{\text{red}}(n) : n \in \mathbb{N})$ is a $2$-regular sequence, we prove an explicit evaluation for the reduced factor 
 complexity function $\rho_{\text{{\bf f}}}^{\text{red}}$ for the (regular) paperfolding sequence $\text{{\bf f}}$, together with an evaluation 
 for the reduced abelian complexity function $\rho_{\text{{\bf f}}}^{\text{ab}, \text{red}}$ for $\text{{\bf f}}$. We conclude with open 
 problems concerning $\rho_{\text{{\bf t}}}^{\text{ab}, \text{red}}$. 
\end{abstract}


\vspace{0.1in}

\noindent {\footnotesize \emph{Keywords:} factor complexity, automatic sequence, 
 integer sequence, subword, Thue--Morse sequence, paperfolding sequence, abelian complexity, 
 regular sequence, {\tt Walnut}}

\section{Introduction}
 For an infinite sequence $\text{{\bf x}}$ over a finite alphabet, the \emph{factor complexity function} 
\begin{equation}\label{factorx}
 \rho_{\text{{\bf x}}}\colon \mathbb{N}_{0} \to \mathbb{N} 
\end{equation}
 maps $n \geq 0$ to the number of distinct factors of $\text{{\bf x}}$ of length $n$. Similarly, the \emph{abelian complexity function}
\begin{equation}\label{abelianx}
 \rho^{\text{ab}}_{\text{{\bf x}}}\colon \mathbb{N}_{0} \to \mathbb{N} 
\end{equation}
 maps $n \geq 0$ to the number of distinct factors of $\text{{\bf x}}$ of length $n$, up to equivalence by rearrangements of characters 
 (possibly by an identity permutation) in words of the same length. To construct simplified versions of the complexity functions in 
 \eqref{factorx} and \eqref{abelianx} and of other previously introduced complexity functions, we introduce counting functions based 
 on factors that are reduced according to a function on nonempty, finite words defined below. 

 We henceforth adopt the convention whereby infinite sequences are denoted in boldface and whereby finite sequences may be referred to 
 as (finite) \emph{words} and are denoted without boldface. For a nonempty and finite word $w$, let $w$ be written as $$ w = c_1 
 \cdots c_1 c_2 \cdots c_2 \cdots c_n \cdots c_n $$ for some positive integer $n$, where $c_1$, $c_2$, $\ldots$, $c_n$ are characters such 
 that $c_{i} \neq c_{i+1}$ for all possible indices $i$. A maximal block of identical characters, $c_i \cdots c_i$ is called a \emph{run}.
 We then define the \emph{reduction}   $\text{red}(w)$ of $w$ so that $$ \text{red}(w) = 
 c_1 c_2 \cdots c_n. $$ For finite words $w$ and $v$, we say that $w$ and $v$ are \emph{reduced-equivalent},  writing $w 
 \sim_{\text{red}} v$, if $\text{red}(w) = \text{red}(v)$.  In the particular case when $w$ is over the binary alphabet $\{0,1\}$,
 we see that $\text{red}(w)$ must consist of alternating $0$'s and $1$'s and hence that the $\sim_{\text{red}}$ equivalence class of $w$
 is uniquely determined by the first letter of $w$ and the number of runs in $w$.  For example, if $w=0010110$ and $v=0111010$ then
 $\text{red}(w)=\text{red}(v)=01010$ and $w \sim_{\text{red}} v$.

 For an infinite sequence $\text{{\bf x}}$ over a finite alphabet,  we then define the \emph{reduced factor complexity function} 
\begin{equation}\label{redfactorx}
 \rho_{\text{{\bf x}}}^{\text{red}}\colon \mathbb{N}_{0} \to \mathbb{N} 
\end{equation}
 by analogy with \eqref{factorx} so that \eqref{redfactorx} maps   $n \geq 0$ to the number of distinct factors of $\text{{\bf x}}$ of  
  length $n$,   up to equivalence by $\sim_{\text{red}}$. Similarly,   we define the \emph{reduced abelian complexity function} 
\begin{equation}\label{abredx}
 \rho_{\text{{\bf x}}}^{\text{ab}, \text{red}}\colon \mathbb{N}_{0} \to \mathbb{N} 
\end{equation}
 by analogy with \eqref{abelianx} so that \eqref{abredx}
 maps $n \geq 0$ to the number of distinct factors of $\text{{\bf x}}$ of length $n$, 
 where two factors $v$ and $w$ are considered to be equivalent 
 if $\text{red}(v)$ is of the same length as $\text{red}(w)$
 and $\text{red}(v)$ can be obtained by rearranging the characters of $\text{red}(w)$ 
 (possibly by an identity permutation). 

   Informally, given a previously introduced complexity function $p$ defined on $\text{{\bf x}}$,    a \emph{reduced complexity function}  
 associated with $p$ may   be defined by counting length-$n$ factors of $\text{{\bf x}}$, up to an equivalence relation   such that factors  
 $v$ and $w$ of equal length are equivalent if $\text{red}(v)$ and $\text{red}(w)$   are equivalent according to how finite words are  
  enumerated in the definition of $p$.   The reduced abelian complexity function defined above provides a prototypical instance of this. 

 For an infinite sequence denoted as an infinite word, 
 let the index of the initial term be $1$. 
 For $n \in \mathbb{N}$, set $\text{{\bf t}}_{n}$ to be equal to the number of 
 $1$'s, modulo $2$, in the base-$2$ expansion 
 of $n-1$. This produces the famous \emph{Thue--Morse sequence}
\begin{equation}\label{numericaltm}
 \text{{\bf t}} = 011010011001011010010110011010011001011001101001011010\cdots 
\end{equation}
 with reference to the work of Allouche and Shallit on the ubiquitous nature of the sequence 
 in \eqref{numericaltm}~\cite{AlloucheShallit1999}. 
 For a positive integer $n$, we write $n = n' 2^k$ for an odd integer $n'$, 
 and we then set 
 $\text{{\bf f}}_{n} = 0$
 if $n' \equiv 1 \pmod{4}$ 
 and $\text{{\bf f}}_{n} = 1$ otherwise. 
 This allows us to define the \emph{ordinary} (or \emph{regular}) \emph{paperfolding sequence} 
\begin{equation}\label{numericalf}
 \text{{\bf f}} = 
 0010011000110110001001110011011000100110001101110010011\cdots
\end{equation}
 The sequences in \eqref{numericaltm} and \eqref{numericalf}
 may be seen as fundamentally important and 
 prototypical instances of automatic sequences, with reference to the standard monograph on automatic
 sequences~\cite{AlloucheShallit2003}. 
 Past research on the evaluation of the factor complexity function 
 $\rho_{\text{{\bf t}}}$~\cite{Avgustinovich1994,Brlek1989,deLucaVarricchio1989}  
 together with the work of Allouche~\cite{Allouche1992}
 on the evaluation of $ \rho_{\text{{\bf f}}}$ 
 and together with the work of Madill and Rampersad~\cite{MadillRampersad2013} 
 on the evaluation of 
 $\rho_{\text{{\bf f}}}^{\text{ab}}$ 
 motivate problems concerning the evaluation 
 of reduced complexity functions for both $\text{{\bf t}}$ and $\text{{\bf f}}$. 

As noted earlier, the reduced complexity function of an infinite sequence is closely
related to the structure of the runs in the sequence.  The sequence of run lengths
of both $\text{{\bf t}}$ and $\text{{\bf f}}$ have previously been studied.
For instance, Allouche, Allouche, and Shallit~\cite{AlloucheAlloucheShallit2006}
showed that the run length sequence of $\text{{\bf t}}$, given by the fixed point
of the map $1 \to 121, 2 \to 12221$, is not an automatic sequence.
The run length sequence of $\text{{\bf f}}$ has also been studied by
Bunder, Bates, and Arnold~\cite{BunderBatesArnold2024} as well as by Shallit~\cite{Shallit2025}.

\section{Main results}
 We begin by considering the reduced factor 
 complexity functions for $\text{{\bf t}}$ and $\text{{\bf f}}$, 
 and we then consider the reduced abelian complexity function for $\text{{\bf f}}$.
 We conclude in Section \ref{sectionConclusion} with open 
 problems concerning the reduced abelian complexity function for $\text{{\bf t}}$. 

\subsection{Reduced factor complexity functions}
 For convenience, we typically disregard the trivial case 
 whereby a complexity function may or may not count the unique word of length $0$, i.e., 
 the empty or null word. 
 The integer sequence 
\begin{equation}\label{numericalfactortm}
 \big( \rho_{\text{{\bf t}}}(n) : n \in \mathbb{N} \big) 
 = (2, 4, 6, 10, 12, 16, 20, 22, 24, 28, 32, 36, 40, 42, 44, \ldots) 
\end{equation}
 is indexed in the On-Line Encyclopedia of Integer Sequences~\cite{oeis} 
 as {\tt A005942} and satisfies the recursion 
 \[ \rho_{\text{{\bf t}}}(n) = \begin{cases} 
 \rho_{\text{{\bf t}}}\left( \frac{n}{2} \right) + 
 \rho_{\text{{\bf t}}}\left( \frac{n}{2} + 1 \right) & \text{if $n$ is even}, \\ 
 2 \rho_{\text{{\bf t}}}\left( \frac{n + 1}{2} \right) & \text{otherwise}, 
 \end{cases} 
 \] 
 for $n \geq 1$. In constrast to \eqref{numericalfactortm}, 
 we find that the integer sequence 
 $$ \big( \rho^{\text{red}}_{\text{{\bf t}}}(n) : n \in \mathbb{N} \big) = 
 (2, 4, 4, 6, 4, 6, 6, 6, 4, 6, 6, 8, 6, 8, 6, 6, 4, 6, 6, 8, 6, 8, 8, \ldots) $$ 
 is not currently included in the OEIS, and this suggests that the concept of ``reduced complexity'' functions 
 has not been considered previously, and this motivates the recurrence introduced below. 

\begin{theorem}\label{tm_red}
 For every positive integer $n$, we have that 
 \[ \rho^{\text{\emph{red}}}_{\text{{\bf t}}}(n) = \begin{cases} 
 \rho^{\text{\emph{red}}}_{\text{{\bf t}}}\left( \frac{n+1}{2} \right) & \text{if $n$ is odd}, \\ 
 \rho_{\text{{\bf t}}}^{\text{\emph{red}}}(m+1) + 2 & \text{if $n = 4m$ or $n = 4m+2$}. 
 \end{cases} 
 \] 
\end{theorem}

In order to prove this theorem, we recall that the Thue-Morse word ${\bf t}$ can also be defined
as a fixed point of the morphism $\mu$ that maps $0 \to 01$ and $1 \to 10$; i.e., we have
${\mathbf t}=\mu^\omega(0).$ Let ${\cal F}$ be the set of factors of ${\mathbf t}$. For a word $v$, we denote by $^-v$ (resp., $v^-$, $^-v^-$) the result of deleting the first letter (resp., last letter, first and last letters) of $v$.
If $w\in{\cal F}$ and $|w|\ge 4$ then the index $\iota(w)$ of $w$ in ${\mathbf t}$ is fixed modulo 2, and $w$ can be parsed uniquely as one of
$$w=\left\{
\begin{array}{lll}
\mu(u),&\text{ some }u\in{\cal F},\\
^-\mu(u)^-,&\text{ some }u\in{\cal F},\\
\mu(u)^-,&\text{ some }u\in{\cal F},\text{ or}\\
^-\mu(u),&\text{ some }u\in{\cal F}\\
\end{array}
\right.
$$
with these cases corresponding, in order, to the situations where $\langle |w|,\iota(w) \rangle$ modulo 2 is $\langle 0,0 \rangle,\langle 0,1 \rangle,\langle 1,0 \rangle$, or $ \langle1,1 \rangle$.

Call $01$ and $10$ {\em alternations}. By the number of alternations in a word $w$ we mean $w_{\{01,10\}}$, the total number of occurrences of $01$ and $10$ in $w$. 

\begin{lemma}
If $w\in\{0,1\}^*$ contains $\alpha$ alternations, then $\mu(w)$ contains $2|w|-1-\alpha$ alternations.
\end{lemma}
\begin{proof} From the definition of $\mu$, each letter of $w$ maps to an alternation. For each of the $|w|-1$ length 2 factors $ab$ of $w$ there is an alternation in $\mu(w)$ formed from the last letter of $\mu(a)$ and the first letter of $\mu(b)$ exactly when $ab$ is not itself an alternation. Thus $\mu(w)$ has $|w|+(|w|-1-\alpha)=2|w|-1-\alpha$ alternations.
\end{proof}

For a non-negative integer $n$, let $m_n$ (resp., $M_n$) be the minimum (resp., maximum) number of alternations in any $w\in {\cal F}$ with $|w|=n$.

\begin{lemma}\label{tm_max_min} For $n\ge 2$ we have
\begin{align*}
 m_{2n} & = 2n-1-M_{n+1}, \\ 
 M_{2n} & = 2n-1-m_n, \\ 
 m_{2n+1} & = 2 n-M_{n+1}, \text{and} \\
 M_{2n+1} & = 2n-m_{n+1}.
\end{align*}
\end{lemma}

\begin{proof}
 We begin by showing that $m_{2 n}=2n-1-M_{n+1}$. Suppose that $n\ge 2$ and $w$ is a word of length $2n$ with the 
 fewest alternations. Suppose $w$ has even index $\iota$. Sliding one place to the right, consider the word $w'$ of length $2n$ and 
 odd index $\iota+1$. Comparing the count of alternations in $w'$ with that in $w$, we see that $w'$ omits the alternation formed by 
 the first two letters of $w$, and may contain at most one new alternation, as a suffix. Thus $w'$ can contain no more alternations than 
 $w$. Thus $w'$ must be a word of length $2n$ with the fewest alternations.

Write $w'=^-\mu(u)^-$ for some $u\in{\cal F}$ with $|u|=n+1$. Let $u$ contain $\alpha$ alternations. Thus $w'$ contains $2|u|-1-\alpha$ alternations. We claim that $u$ has the maximum number of alternations for a word of length $n+1$, so that $M_{n+1}=\alpha$. 

To get a contradiction, suppose that $u'\in{\cal F}$ with $|u'|=n+1$ contains $\beta$ alternations where $\beta>\alpha$. Then $\mu(u')$ contains $2|u'|-1-\beta=2n+1-\beta$ alternations. The word $^-\mu(u')^-$ omits the first and last alternations in $\mu(u')$, and thus contains
$2n-1-\beta<2n-1-\alpha$ alternations, i.e., fewer than $w'$. Since $|^-\mu(u')^-|=2n$, this contradicts the choice of $w'$.
We conclude that $M_{n+1}=\alpha$, which means that
$m_{2n}=2n-1-M_{n+1}$, as claimed. 

Next we show that $M_{2n}=2n-1-m_n$. Suppose that $n\ge 2$ and $w$ is a word of length $2n$ with the most alternations. Suppose $w$ has odd index $\iota$. Sliding one place to the right, consider the word $w'$ of length $2n$ and even index $\iota+1$. Comparing the count of alternations in $w'$ with that in $w$, we see that $w'$ may possibly have lost an alternation formed by the first two letters of $w$, but definitely contains a new alternation, as a suffix. Thus $w'$ contains at least as many alternations as $w$. Thus $w'$ must also be 
a word of length $2n$ with the most alternations.

Write $w'=\mu(u)$ for some $u\in{\cal F}$ with $|u|=n$. Let $u$ contain $\alpha$ alternations. Thus $w'$ contains $2|u|-1-\alpha=2n-1-\alpha $ alternations. We claim that $u$ has the minimum number of alternations for a word of length $n$, so that $m_n=\alpha$. 

 To get a contradiction, suppose that $u'\in{\cal F}$ with $|u'|=n$ contains $\beta$ alternations where $\beta<\alpha$. Then $\mu(u')$
 has length $2n$ and contains $2|u'|-1-\beta=2n-1-\beta>2n-1-\alpha$ alternations, i.e., more than $w'$. This contradicts the choice
 of $w'$. We conclude that $m_n=\alpha$, which means that $M_{2n}=2n-1-m_n$, as claimed. 

The third equality is proved analogously:
Suppose that $n\ge 2$ and $w$ is a word of length $2n+1$ with the fewest alternations. 
Write $w=\mu(u)^-$ or $w=^-\mu(u)$ for some $u\in{\cal F}$ with $|u|=n+1$. Let $u$ contain $\alpha$ alternations. Since $w$ is obtained by deleting a letter from one end of $\mu(u)$, it has one fewer alternation than $\mu(u)$. Thus $w$ contains $2|u|-1-\alpha-1=
2n-\alpha$ alternations. We claim that $u$ has the maximum number of alternations for a word of length $n+1$, so that $M_{n+1}=\alpha$. 

 To get a contradiction, suppose that $u'\in{\cal F}$ with $|u'|=n$ contains $\beta$ alternations where $\beta>\alpha$. Then
 $\mu(u')^-$ is a word of length $2n+1$ that contains $2|u'|-1-\beta-1=2n-\beta<2n-\alpha$ alternations. This contradicts the
 choice of $w'$. We thus conclude that $M_{n+1}=\alpha$, which means that
$m_{2n+1}=2n-M_{n+1}$, as claimed. 

 Finally, suppose that $n\ge 2$ and $w$ is a word of length $2n+1$ with the most alternations. 
Write $w=\mu(u)^-$ or $w=^-\mu(u)$ for some $u\in{\cal F}$ with $|u|=n+1$. Let $u$ contain $\alpha$ alternations. Since $w$ is obtained by deleting a letter from one end of $\mu(u)$, it has one fewer alternation than $\mu(u)$. Thus $w$ contains $2|u|-1-\alpha-1=
2n-\alpha$ alternations. We claim that $u$ has the minimum number of alternations for a word of length $n+1$, so that $m_{n+1}=\alpha$. 

 To get a contradiction, suppose that $u'\in{\cal F}$ with $|u'|=n$ contains $\beta$ alternations where $\beta<\alpha$. Then
 $\mu(u')^-$ is a word of length $2n+1$ that contains $2|u'|-1-\beta-1=2n-\beta>2n-\alpha$ alternations. This contradicts 
 the choice of $w'$.
 Consequently, we may 
 conclude that $m_{n+1}=\alpha$, which means that
$M_{2n+1}=2n-m_{n+1}$, as claimed. 
\end{proof}

\begin{lemma}\label{tm_mod4} 
 For $n\ge 1$ we have
\begin{align*}
 m_{4n} & = 2n-1+m_{n+1}, \\
M_{4n} & = 2n+M_{n+1}, \\
 m_{4n+2} & = 2 n+m_{n+1}, \text{and} \\
M_{4n+2}&=2n+1+M_{n+1}. 
\end{align*}
\end{lemma}

\begin{proof}
To get each identity, we apply Lemma~\ref{tm_max_min} twice. We have
\begin{align*}
m_{4n} &= 4n-1-M_{2n+1}\\
&= 4n-1-(2n-m_{n+1})\\
&= 2n-1+m_{n+1},
\end{align*}
\begin{align*}
M_{4n} &= 4n-1-m_{2n}\\
&= 4n-1-(2n-1-M_{n+1})\\
&= 2n+M_{n+1},
\end{align*}
\begin{align*}
m_{4n+2} &= 4n+2-1-M_{2n+2}\\
&= 4n+1-(2n+2-1-m_{n+1})\\
&= 2n+m_{n+1},
\end{align*}
and
\begin{align*}
M_{4n+2} &= 4n+2-1-m_{2n+1}\\
&= 4n+1-(2n-M_{n+1})\\
&= 2n+1+M_{n+1}.
\end{align*}
\end{proof}

We can now give the proof of Theorem~\ref{tm_red}.

\begin{proof}[(Proof of Theorem~\ref{tm_red})]
Recall that the $\sim_{\text{red}}$ equivalence class of a binary word $w$ is
determined by the first letter of $w$ and the numbers of runs in $w$.
Moreover, the number of runs in $w$ is one greater than the number of
alternations in $w$. Now let $w$ be a factor of length $n$ of ${\bf t}$ with
$i$ runs. Then $i \in [m_n+1, M_n+1]$. Furthermore, for every $i$ in this interval,
there is such a $w$ in ${\bf t}$ (sliding $w$ to the left or right in ${\bf t}$ can
only increase or decrease the number of runs by at most $1$). Since ${\bf t}$ is closed under complement,
we see that
\[
\rho_{\text{{\bf t}}}^{\text{red}}(n) = 2(M_n+1-m_n-1+1) = 2(M_n-m_n+1).
\]
We now consider different cases for $n$ and apply Lemmas~\ref{tm_max_min} and~\ref{tm_mod4}.

If $n$ is odd, say $n=2m+1$, we have
\begin{align*}
\rho_{\text{{\bf t}}}^{\text{red}}(n) &= 2(M_{2m+1}-m_{2m+1}+1)\\
&= 2(2m-m_{m+1}-2m+M_{m+1}+1)\\
&= 2(M_{m+1}-m_{m+1}+1)\\
&= \rho_{\text{{\bf t}}}^{\text{red}}(m+1)\\
&= \rho_{\text{{\bf t}}}^{\text{red}}\left(\frac{n+1}{2}\right).
\end{align*}

If $n=4m$, we have
\begin{align*}
\rho_{\text{{\bf t}}}^{\text{red}}(n) &= 2(M_{4m}-m_{4m}+1)\\
&= 2(2m+M_{m+1}-2m+1-m_{m+1}+1)\\
&= 2(M_{m+1}-m_{m+1}+1)+2\\
&= \rho_{\text{{\bf t}}}^{\text{red}}(m+1)+2.
\end{align*}

If $n=4m+2$, we have
\begin{align*}
\rho_{\text{{\bf t}}}^{\text{red}}(n) &= 2(M_{4m+2}-m_{4m+2}+1)\\
&= 2(2m+1+M_{m+1}-2m-m_{m+1}+1)\\
&= 2(M_{m+1}-m_{m+1}+1)+2\\
&= \rho_{\text{{\bf t}}}^{\text{red}}(m+1)+2.
\end{align*}
\end{proof}

 From Theorem \ref{tm_red}, we thus have that 
 the integer sequence
 $( \rho^{\text{red}}_{\text{{\bf t}}}(n) : n \in \mathbb{N} )$ is a $2$-regular sequence, 
 referring to Allouche and Shallit's text~\cite[\S16]{AlloucheShallit2003} for background on $k$-regular sequences. 

 The integer sequence 
\begin{equation*}
 \big( \rho_{\text{{\bf f}}}(n) : n \in \mathbb{N} \big) 
 = ( 2, 4, 8, 12, 18, 23, 28, 32, 36, 40, 44, 48, 52, 56, 60, \ldots) 
\end{equation*}
 is indexed in the OEIS as {\tt A337120} and is such that 
\begin{equation}\label{Allouche4n}
 \rho_{\text{{\bf f}}}(n) = 4 n 
\end{equation}
 for all integers $n$ greater than $6$, as proved in 1992 by Allouche~\cite{Allouche1992}. The integer sequence 
 \begin{equation*}
 \big( \rho_{\text{{\bf f}}}^{\text{red}}(n) : n \in \mathbb{N} \big) 
 = ( 2, 4, 6, 4, 6, 4, 6, 4, 4, 4, 6, 4, 6, 4, 6, 4, 4, 4, 6, 4, 6, 4, 6, \ldots) 
\end{equation*}
 is not currently given in the OEIS, further suggesting that our notion of reduced complexity functions is original. 
 In contrast to \eqref{Allouche4n}, we find that 
 $\rho_{\text{{\bf f}}}^{\text{red}}$ is eventually periodic, as below. 

\begin{theorem}\label{pf_red}
 For every positive integer $n$, we have that 
 \[ \rho^{\text{\emph{red}}}_{\text{{\bf f}}}(n) = \begin{cases} 
 4 & \text{if $n \equiv \{ 0, 1, 2, 4, 6 \} \pmod{8}$}, \\ 
 6 & \text{if $n \equiv \{ 3, 5, 7 \} \pmod{8}$}. 
 \end{cases} 
 \] 
\end{theorem}

 The proof is an application of several lemmas given below that treat different cases for the value of $n$ modulo $8$. We also make 
 use of the following alternative construction of the paperfolding sequence, which is known as the \emph{Toeplitz construction}:
\begin{itemize}
\item
 Start with an infinite sequence of \emph{gaps}, denoted by ?.
 \[\begin{array}{*{16}{c}}
 ? & ? & ? & ? & ? & ? & ? & ? & ? & ? & ? & ? & ? & ? & ? & \cdots
 \end{array}\]
\item
 Fill every other gap with alternating $0$'s and $1$'s.
 \[\begin{array}{*{16}{c}}
 0 & ? & 1 & ? & 0 & ? & 1 & ? & 0 & ? & 1 & ? & 0 & ? & 1 & \cdots
 \end{array}\]
\item
 Repeat.
 \[\begin{array}{*{16}{c}}
 0 & 0 & 1 & ? & 0 & 1 & 1 & ? & 0 & 0 & 1 & ? & 0 & 1 & 1 & \cdots
 \end{array}\]
 \[\begin{array}{*{16}{c}}
 0 & 0 & 1 & 0 & 0 & 1 & 1 & ? & 0 & 0 & 1 & 1 & 0 & 1 & 1 & \cdots
 \end{array}\]
 \[\begin{array}{*{16}{c}}
 0 & 0 & 1 & 0 & 0 & 1 & 1 & 0 & 0 & 0 & 1 & 1 & 0 & 1 & 1 & \cdots
 \end{array}\]
\end{itemize}
In the limit, one obtains the paperfolding word ${\bf f}$.

\begin{lemma}\label{odd_len}
 Let $a,b\in\{0,1\}$, $a \neq b$, and let $w$ be a length-$(2n+1)$ factor of ${\bf f}$ of the form
 \[w=a?b?a?b?\cdots a?b \text{ or } w=a?b?a?b?\cdots b?a.\]
 Then $w$ has $n+1$ runs.
\end{lemma}

\begin{proof}
 The number of runs in $w$ is one more than the number of occurrences of $ab$ and $ba$ in $w$.
 It is clear that there are $n$ such occurrences.
\end{proof}

\begin{lemma}\label{f_2n}
 For every positive integer $n$, we have that $\rho^{\text{\emph{red}}}_{\text{{\bf f}}}(2n) = 4.$
\end{lemma}

\begin{proof}
 Let $w$ be a factor of ${\bf f}$ of length $2n$. Then $w$ has one of the forms
 \[
 \begin{array}{ll}
 w=a?b?a?b?\cdots a?bx, & w=a?b?a?b?\cdots b?ax,\\
 w=xa?b?a?b?\cdots a?b, & w=xa?b?a?b?\cdots b?a,
 \end{array}
 \]
 for some $a,b,x \in \{0,1\}$ with $a \neq b$.

 Suppose $w=a?b?a?b?\cdots a?bx$. By Lemma~\ref{odd_len}, the word $w$ has $n$ runs if $x=b$
 and $n+1$ runs if $x=a$. We need to show that all four choices for $a,x \in \{0,1\}$ occur.
 Let $i$ be the position of the $x$ in some occurrence of $w$ in ${\bf p}$.
 
 If $a=0$ then $i \equiv 0 \pmod{4}$. Choose $i \equiv 4 \pmod{16}$ to get $x=0$ and choose
 $i \equiv 12 \pmod{16}$ to get $x=1$.

 If $a=1$ then $i \equiv 2 \pmod{4}$. Choose $i \equiv 2 \pmod{8}$ to get $x=0$ and choose
 $i \equiv 6 \pmod{8}$ to get $x=1$.

 Hence $w$ has four possible $\sim_{\text{red}}$ equivalence classes and all occur in ${\bf p}$.
 A similar analysis applies to the other cases for $w$ and yields four equivalence classes in
 every case.
\end{proof}

\begin{lemma}\label{f_1mod8}
 For every positive integer $n$, we have that $\rho^{\text{\emph{red}}}_{\text{{\bf f}}}(8n+1) = 4.$
\end{lemma}

\begin{proof}
 Let $w$ be a factor of ${\bf p}$ of length $8n+1$. Then $w$ has one of the forms
 \[
 \begin{array}{ll}
 w=a?b?a?b?\cdots b?a, & w=xa?b?a?b?\cdots a?by,
 \end{array}
 \]
 for some $a,b,x,y \in \{0,1\}$ with $a \neq b$.

 Suppose that $w=a?b?a?b?\cdots b?a$. By Lemma~\ref{odd_len}, the word $w$ has $4n+1$
 runs; the two possibilities for $a\in\{0,1\}$ give two equivalence classes.

 Now suppose that $w=xa?b?a?b?\cdots a?by$ and let $i$ and $j$ be the positions
 of the $x$ and $y$ in some occurrence of $w$ in ${\bf p}$. Note that $j-i=8n$.
 If we write $w=xw'y$, we can apply Lemma~\ref{odd_len} to $w'$ to conclude that $w'$ has $4n$ runs.
 Note that if $x=y$ then $w$ has $4n+1$ runs, but we have already accounted for
 these equivalence classes in the previous case. We therefore consider the case where
 $x \neq y$.

 If $a=0$ then $i \equiv 0 \pmod{4}$. Write $8n=(2n'+1)2^k$.
 To get $xy=01$, choose $i=2^{k-1}$, which gives $j=(4n'+3)2^{k-1}$.
 In this case $w$ has $4n$ runs. 
 To get $xy=10$, choose $i=3\cdot 2^{k-1}$, which gives $j=(4n'+5)2^{k-1}$.
 In this case $w$ has $4n+2$ runs.

 If $a=1$ then $i \equiv 2 \pmod{4}$. We therefore have either $i \equiv j \equiv 2 \pmod{8}$
 or $i \equiv j \equiv 6 \pmod{8}$. Both cases give $x=y$, a contradiction.

 In total then, $w$ has four possible $\sim_{\text{red}}$ equivalence classes and all occur in ${\bf p}$.
\end{proof} 

 For the remaining equivalence classes modulo $8$, we make use of the software {\tt Walnut}~\cite{Shallit2023book} in the proofs of the 
 following lemmas. Note that, while we have been indexing the terms of the paperfolding word starting with $1$, {\tt Walnut} expects all 
 automatic sequences to be indexed starting with $0$. The indices in the {\tt Walnut} formulas used in the proofs below therefore are $1$ 
 less than you would expect from the surrounding analysis.

\begin{lemma}\label{f_3mod8}
 For every non-negative integer $n$, we have that $\rho^{\text{\emph{red}}}_{\text{{\bf f}}}(8n+3) = 6.$
\end{lemma}

\begin{proof}
 Let $w$ be a factor of ${\bf p}$ of length $8n+3$. Then $w$ has one of the forms
 \[
 \begin{array}{ll}
 w=a?b?a?b?\cdots a?b, & w=xa?b?a?b?\cdots b?ay,
 \end{array}
 \]
 for some $a,b,x,y \in \{0,1\}$ with $a \neq b$.

 Suppose that $w=a?b?a?b?\cdots a?b$. By Lemma~\ref{odd_len}, the word $w$ has $4n+2$
 runs; the two possibilities for $a\in\{0,1\}$ give two equivalence classes.

 Now, suppose that $w=xa?b?a?b? \cdots b?ay$ and let $i$ be the position of the $x$ in some occurrence of $w$ in ${\bf p}$. Note 
 that if $a=0$, then $i \equiv 0 \pmod{4}$, and if $a=1$, then $i \equiv 2 \pmod{4}$. Also, if $x \neq y$, then $w$ has $4n+2$ 
 runs, and we have already accounted for these equivalence classes above. We therefore consider $x=y$. There are two possibilities for 
 $a$ and two for $x$; if these all occur in ${\bf p}$ we get four additional equivalence classes. We verify this with the following 
 {\tt Walnut} commands, which all return TRUE:

 \begin{verbatim}
eval pf_red_3mod8a "?msd_2 An Er P[4*r+3]=@0 & P[4*r+4]=@0 &
 P[4*r+8*n+5]=@0":
eval pf_red_3mod8b "?msd_2 An Er P[4*r+3]=@1 & P[4*r+4]=@0 &
 P[4*r+8*n+5]=@1":
eval pf_red_3mod8c "?msd_2 An Er P[4*r+1]=@0 & P[4*r+2]=@1 &
 P[4*r+8*n+3]=@0":
eval pf_red_3mod8d "?msd_2 An Er P[4*r+1]=@1 & P[4*r+2]=@1 &
 P[4*r+8*n+3]=@1":
 \end{verbatim}
\end{proof}

\begin{lemma}\label{f_5mod8}
 For every non-negative integer $n$, we have that $\rho^{\text{\emph{red}}}_{\text{{\bf f}}}(8n+5) = 6.$
\end{lemma}

\begin{proof}
 Let $w$ be a factor of ${\bf p}$ of length $8n+5$. Then $w$ has one of the forms
 \[
 \begin{array}{ll}
 w=a?b?a?b?\cdots b?a, & w=xa?b?a?b?\cdots a?by,
 \end{array}
 \]
 for some $a,b,x,y \in \{0,1\}$ with $a \neq b$.

 Suppose that $w=a?b?a?b?\cdots b?a$. By Lemma~\ref{odd_len}, the word $w$ has $4n+3$
 runs; the two possibilities for $a\in\{0,1\}$ give two equivalence classes.

 Now suppose that $w=xa?b?a?b?\cdots a?by$. In this case we get our four additional
 equivalence classes when $x \neq y$. We verify their existence for all $n$ with
 the following {\tt Walnut} commands, which all return TRUE:

 \begin{verbatim}
eval pf_red_5mod8a "?msd_2 An Er P[4*r+3]=@0 & P[4*r+4]=@0 &
 P[4*r+8*n+7]=@1":
eval pf_red_5mod8b "?msd_2 An Er P[4*r+3]=@1 & P[4*r+4]=@0 &
 P[4*r+8*n+7]=@0":
eval pf_red_5mod8c "?msd_2 An Er P[4*r+1]=@0 & P[4*r+2]=@1 &
 P[4*r+8*n+5]=@1":
eval pf_red_5mod8d "?msd_2 An Er P[4*r+1]=@1 & P[4*r+2]=@1 &
 P[4*r+8*n+5]=@0":
 \end{verbatim}
\end{proof}

\begin{lemma}\label{f_7mod8}
 For every non-negative integer $n$, we have that $\rho^{\text{\emph{red}}}_{\text{{\bf f}}}(8n+7) = 6.$
\end{lemma}

\begin{proof}
 Let $w$ be a factor of ${\bf p}$ of length $8n+7$. Then $w$ has one of the forms
 \[
 \begin{array}{ll}
 w=a?b?a?b?\cdots a?b, & w=xa?b?a?b?\cdots b?ay,
 \end{array}
 \]
 for some $a,b,x,y \in \{0,1\}$ with $a \neq b$.

 Suppose that $w=a?b?a?b?\cdots a?b$. By Lemma~\ref{odd_len}, the word $w$ has $4n+4$
 runs; the two possibilities for $a\in\{0,1\}$ give two equivalence classes.

 Now suppose that $w=xa?b?a?b?\cdots b?ay$. In this case we get our four additional
 equivalence classes when $x=y$. We verify their existence for all $n$ with
 the following {\tt Walnut} commands, which all return TRUE:

 \begin{verbatim}
eval pf_red_7mod8a "?msd_2 An Er P[4*r+3]=@0 & P[4*r+4]=@0 &
 P[4*r+8*n+9]=@0":
eval pf_red_7mod8b "?msd_2 An Er P[4*r+3]=@1 & P[4*r+4]=@0 &
 P[4*r+8*n+9]=@1":
eval pf_red_7mod8c "?msd_2 An Er P[4*r+1]=@0 & P[4*r+2]=@1 &
 P[4*r+8*n+7]=@0":
eval pf_red_7mod8d "?msd_2 An Er P[4*r+1]=@1 & P[4*r+2]=@1 &
 P[4*r+8*n+7]=@1":
 \end{verbatim}
\end{proof}

This last lemma completes the proof of Theorem~\ref{pf_red}.

\subsection{Reduced abelian complexity functions}
 Relative to our proof of Theorem \ref{pf_red}, a similar approach can be used to evaluate the reduced abelian complexity function for 
 $\text{{\bf f}}$. We later consider the problem of determining a recurrence for the reduced abelian complexity function 
 for $\text{{\bf t}}$. 

 The integer sequence 
\begin{equation}\label{numericalabf}
 \big( \rho_{\text{{\bf f}}}^{\text{ab}}(n) : n \in \mathbb{N} \big) 
 = (2, 3, 4, 3, 4, 5, 4, 3, 4, 5, 6, 5, 4, 5, 4, 3, 4, 5, 6, 5, 6, \ldots) 
\end{equation}
 agrees with the OEIS entry {\tt A214613} and was first shown to be 
 a $2$-regular sequence by Madill and Rampersad~\cite{MadillRampersad2013}. 
 In contrast to \eqref{numericalabf}, we find that 
\begin{equation*}
 \big( \rho_{\text{{\bf f}}}^{\text{ab}, \text{red}}(n) : n \in \mathbb{N} \big) 
 = (2, 3, 5, 3, 4, 3, 5, 3, 4, 3, 5, 3, 4, 3, 5, 3, 4, 3, 5, 3, 4, 3, \ldots) 
\end{equation*}
 is eventually periodic. 
 
\begin{theorem}
 For every positive integer $n$, we have that 
\begin{equation*}
 \rho^{\text{{\emph{ab}}}, \text{{\emph{red}}}}_{\text{{\bf f}}}(n) = \begin{cases} 
 3 & \text{if $n$ is even}, \\ 
 4 & \text{if $n > 1$ and $n \equiv 1 \pmod{4}$}, \\ 
 5 & \text{$n \equiv 3 \pmod{4}$}. 
 \end{cases} 
\end{equation*}
\end{theorem}

\begin{proof}
 Suppose two binary words $w$ and $w'$ begin with different letters but have the same number of runs. If $w$ and $w'$ have an even 
 number of runs then $\text{red}(w)$ and $\text{red}(w')$ are abelian equivalent; whereas, if $w$ and $w'$ have an odd number of runs 
 then $\text{red}(w)$ and $\text{red}(w)$ are not abelian equivalent.

 To prove the result we go through the various lemmas covering the different cases of the proof of Theorem~\ref{pf_red} and examine the 
 parity of the number of runs in each $\sim_{\text{red}}$ equivalence class.

If $n=2m$ then the proof of Lemma~\ref{f_2n} gives:
\begin{itemize}
 \item $2$ $\sim_{\text{red}}$ equivalence classes with $m$ runs, and
 \item $2$ $\sim_{\text{red}}$ equivalence classes with $m+1$ runs.
\end{itemize}
One of $m$ and $m+1$ is even and the two corresponding $\sim_{\text{red}}$ equivalence classes
merge under the abelian reduced equivalence. Hence, we have
$\rho^{\text{{\emph{ab}}}, \text{{\emph{red}}}}_{\text{{\bf f}}}(n) = 3$.

If $n=8m+1$, $m>0$, then the proof of Lemma~\ref{f_1mod8} gives:
\begin{itemize}
 \item $1$ $\sim_{\text{red}}$ equivalence class with $4m$ runs,
 \item $2$ $\sim_{\text{red}}$ equivalence classes with $4m+1$ runs, and
 \item $1$ $\sim_{\text{red}}$ equivalence class with $4m+2$ runs.
\end{itemize}
Hence, we have $\rho^{\text{{\emph{ab}}}, \text{{\emph{red}}}}_{\text{{\bf f}}}(n) = 4$.

If $n=8m+3$ then the proof of Lemma~\ref{f_3mod8} gives:
\begin{itemize}
 \item $2$ $\sim_{\text{red}}$ equivalence classes with $4m+1$ runs,
 \item $2$ $\sim_{\text{red}}$ equivalence classes with $4m+2$ runs, and
 \item $2$ $\sim_{\text{red}}$ equivalence classes with $4m+3$ runs.
\end{itemize}
Hence, we have $\rho^{\text{{\emph{ab}}}, \text{{\emph{red}}}}_{\text{{\bf f}}}(n) = 5$.

If $n=8m+5$ then the proof of Lemma~\ref{f_5mod8} gives:
\begin{itemize}
 \item $2$ $\sim_{\text{red}}$ equivalence classes with $4m+2$ runs,
 \item $2$ $\sim_{\text{red}}$ equivalence classes with $4m+3$ runs, and
 \item $2$ $\sim_{\text{red}}$ equivalence classes with $4m+4$ runs.
\end{itemize}
Hence, we have $\rho^{\text{{{ab}}}, \text{{{red}}}}_{\text{{\bf f}}}(n) = 4$.

If $n=8m+7$ then the proof of Lemma~\ref{f_7mod8} gives:
\begin{itemize}
 \item $2$ $\sim_{\text{red}}$ equivalence classes with $4m+3$ runs, and
 \item $2$ $\sim_{\text{red}}$ equivalence classes with $4m+4$ runs, and
 \item $2$ $\sim_{\text{red}}$ equivalence classes with $4m+5$ runs.
\end{itemize}
Hence, we have $\rho^{\text{{{ab}}}, \text{{{red}}}}_{\text{{\bf f}}}(n) = 5$.
\end{proof}

 The problem of determining a recursion for 
\begin{equation}\label{numericalproblem}
 \big( \rho^{\text{ab}, \text{red}}_{\text{{\bf t}}}(n) : n \in \mathbb{N} 
 \big) = (2, 3, 3, 4, 3, 5, 4, 5, 3, 4, 5, 6, 4, 6, 5, 4, 3, 5, 4, \ldots)
\end{equation}
 appears to be much more challenging, relative to the above Theorems. 
 This problem is motivated by 
 past research on the abelian complexity functions for Thue--Morse-like sequences~\cite{BlanchetSadriCurrieRampersadFox2014,ChenWen2019,Greinecker2015,KaboreKientega2017,ParreauRigoRowlandVandomme2015}, 
 and leads us to provide, in the below section, 
 open problems concerning $\rho_{\text{{\bf t}}}^{\text{ab}, \text{red}}$. 

\section{Conclusion}\label{sectionConclusion}
 Although it appears that 
 $$ \rho^{\text{{ab}}, \text{{red}}}_{\text{{\bf t}}}(2n + 1) = 
 \rho^{\text{{ab}}, \text{{red}}}_{\text{{\bf t}}}\left( n + 1 \right) $$ 
 for nonnegative integers $n$, 
 the problem of determining a full recursion for 
 $ \rho^{\text{\emph{ab}}, \text{\emph{red}}}_{\text{{\bf t}}}(n) $ 
 seems to be challenging. 
 It appears that 
\begin{equation}\label{displayconjecture}
 \left| \rho^{\text{{ab}}, \text{{red}}}_{\text{{\bf t}}}(4n+2) 
 - \rho^{\text{{ab}}, \text{{red}}}_{\text{{\bf t}}}(4n) \right| = \begin{cases} 
 0 & \text{if $ \text{{\bf t}}_{n+1} = \text{{\bf t}}_{3n+1}$}, \\ 
 1 & \text{otherwise}, 
 \end{cases} 
\end{equation}
 but it is unclear how the sign of 
 $ \rho^{\text{{ab}}, \text{{red}}}_{\text{{\bf t}}}(4n+2) 
 - \rho^{\text{{ab}}, \text{{red}}}_{\text{{\bf t}}}(4n) $ 
 could be evaluated in an explicit way for the nonzero case, and we leave this as an open problem. 
 Also, it is unclear as to how a suitable recursion could be determined for 
 $\rho_{\text{{\bf t}}}^{\text{ab}, \text{red}}(4n)$, and we leave this as an open problem. 
 We also leave it as an open problem to prove \eqref{displayconjecture}. 
 It also appears that the integer sequence in \eqref{numericalproblem} is not $k$-automatic, and we leave it as an open 
 problem to prove this. 

\subsection*{Acknowledgements}
 The first author was supported by an NSERC Discovery Grant 
 and thanks Jean-Paul Allouche for useful feedback.  The
 second and third authors also acknowledge the support of the
 NSERC Discovery Grant program.

 \ 

{\textsc{John M. Campbell}} 

\vspace{0.1in}

Department of Mathematics and Statistics 

Dalhousie University

6283 Alumni Crescent, Halifax, NS B3H 4R2

\vspace{0.1in}

{\tt jh241966@dal.ca}

 \ 

{\textsc{James Currie}} 

\vspace{0.1in}

{\textsc{Narad Rampersad}} 

\vspace{0.1in}

 Department of Mathematics and Statistics 

 University of Winnipeg

 515 Portage Ave, Winnipeg, MB R3B 2E9

\vspace{0.1in}

{\tt j.currie@uwinnipeg.ca}

\vspace{0.1in}

{\tt n.rampersad@uwinnipeg.ca}

\end{document}